\documentclass[11pt,oneside]{amsart}

\usepackage{amsmath,amsthm,verbatim,amssymb,amsfonts,amscd, graphicx}
\usepackage{mathrsfs}
\usepackage[utf8]{inputenc}
\usepackage[retainorgcmds]{IEEEtrantools}
\usepackage{tikz-cd}
\usepackage{graphics}
\usepackage[all]{xy}
\usepackage{incgraph}
\usepackage{mathtools}
\usepackage{graphicx}
\usepackage{marginnote}
\usepackage{comment}
\usepackage{tikz}
\usetikzlibrary{arrows.meta}
\usetikzlibrary{bending}

\usepackage{hyperref}

\theoremstyle{plain}
\newtheorem{theorem}{Theorem}[section]
\newtheorem{corollary}[theorem]{Corollary}
\newtheorem{lemma}[theorem]{Lemma}
\newtheorem{proposition}[theorem]{Proposition}

\theoremstyle{remark}
\newtheorem{remark}[theorem]{Remark}

\def\Hom{\mathrm{Hom}}

\def\makeop#1{\expandafter\def\csname#1\endcsname
  {\mathop{\rm #1}\nolimits}\ignorespaces}
\makeop{Hom}   \makeop{End}   \makeop{Aut}   \makeop{Isom}  \makeop{Pic}
\makeop{Gal}   \makeop{ord}   \makeop{Char}  \makeop{Div}   \makeop{Lie}
\makeop{PGL}   \makeop{Corr}  \makeop{PSL}   \makeop{sgn}   \makeop{Spf}
\makeop{Spec}  \makeop{Map}    \makeop{Nm}    \makeop{Fr}    \makeop{disc}
\makeop{Proj}  \makeop{supp}  \makeop{ker}   \makeop{im}    \makeop{dom}
\makeop{coker} \makeop{Stab}  \makeop{SO}    \makeop{SL}    \makeop{SL}
\makeop{Cl}    \makeop{cond}  \makeop{Br}    \makeop{inv}   \makeop{rank}
\makeop{id}    \makeop{Fil}   \makeop{Frac}  \makeop{GL}    \makeop{SU}
\makeop{Trd}   \makeop{Sp}    \makeop{Tr}    \makeop{Trd}   \makeop{diag}
\makeop{Res}   \makeop{ind}   \makeop{depth} \makeop{Tr}    \makeop{st}
\makeop{Ad}    \makeop{Int}   \makeop{tr}    \makeop{Sym}   \makeop{can}
\makeop{length}\makeop{SO}    \makeop{torsion} \makeop{GSp} \makeop{Ker}
\makeop{Adm}   \makeop{Frob}  \makeop{id}    \makeop{Tor}   \makeop{Ind}
\makeop{CoInd} \makeop{Inf}   \makeop{Cor}   \makeop{CoInf} \makeop{coLie}
\makeop{Ann}   \makeop{Shv}   \makeop{perf}  \makeop{Cont}  \makeop{Perfd}
\makeop{can}

\newcommand{\Z}{\mathbb Z}
\newcommand{\Q}{\mathbb Q}

\newcommand{\C}{\mathbb C}

\newcommand{\D}{\mathcal D}    
\renewcommand{\O}{\mathcal O} 

\newcommand{\Ainf}{{A_{\mathrm{inf}}}} 
\newcommand{\Acris}{A_{\mathrm{cris}}} 

\newcommand{\fS}{\mathfrak{S}}

\newcommand{\e}{\mathfrak e}

\newcommand{\cris}{\mathrm{cris}}

\newcommand{\calR}{{\mathcal R}}

\newcommand{\gr}{\mathrm{gr}}
\newcommand{\tor}{\mathrm{tor}}
\usepackage{relsize}
\usepackage[bbgreekl]{mathbbol}
\usepackage{amsfonts}
\DeclareSymbolFontAlphabet{\mathbb}{AMSb} 
\DeclareSymbolFontAlphabet{\mathbbl}{bbold}

\newcommand{\M}{\mathfrak M}

\newcommand{\dR}{\mathrm{dR}}
\newcommand{\ev}{\mathrm{ev}}

\makeatletter
\newcommand{\colim@}[2]{%
  \vtop{\m@th\ialign{##\cr
    \hfil$#1\operator@font colim$\hfil\cr
    \noalign{\nointerlineskip\kern1.5\ex@}#2\cr
    \noalign{\nointerlineskip\kern-\ex@}\cr}}%
}
\newcommand{\colim}{%
  \mathop{\mathpalette\colim@{\rightarrowfill@\textstyle}}\nmlimits@
}
\makeatother

\makeatletter
\newcommand*{\da@rightarrow}{\mathchar"0\hexnumber@\symAMSa 4B }
\newcommand*{\da@leftarrow}{\mathchar"0\hexnumber@\symAMSa 4C }
\newcommand*{\xdashrightarrow}[2][]{%
  \mathrel{%
    \mathpalette{\da@xarrow{#1}{#2}{}\da@rightarrow{\,}{}}{}%
  }%
}
\newcommand{\xdashleftarrow}[2][]{%
  \mathrel{%
    \mathpalette{\da@xarrow{#1}{#2}\da@leftarrow{}{}{\,}}{}%
  }%
}
\newcommand*{\da@xarrow}[7]{%
  \sbox0{$\ifx#7\scriptstyle\scriptscriptstyle\else\scriptstyle\fi#5#1#6\m@th$}%
  \sbox2{$\ifx#7\scriptstyle\scriptscriptstyle\else\scriptstyle\fi#5#2#6\m@th$}%
  \sbox4{$#7\dabar@\m@th$}%
  \dimen@=\wd0 %
  \ifdim\wd2 >\dimen@
    \dimen@=\wd2 %
26

  \fi
  \count@=2 %
  \def\da@bars{\dabar@\dabar@}%
  \@whiledim\count@\wd4<\dimen@\do{%
    \advance\count@\@ne
    \expandafter\def\expandafter\da@bars\expandafter{%
      \da@bars
      \dabar@ 
    }%
  }%
  \mathrel{#3}%
  \mathrel{%
26

    \mathop{\da@bars}\limits
    \ifx\\#1\\%
    \else
      _{\copy0}%
    \fi
    \ifx\\#2\\%
    \else
      ^{\copy2}%
    \fi
  }%
  \mathrel{#4}%
}
\makeatother

\begin{document}

\author{Tong Liu}
\address[Tong Liu]{Department of Mathematics, Purdue University, 150 N. University Street, West Lafayette, Indiana 47907, USA}
\email{tongliu@math.purdue.edu}

\title{Torsion graded pieces of Nygaard filtration for crystalline representation}
\maketitle
\begin{center}
\dedicatory{Dedicated to the memory of my master advisor Xianke Zhang}    
\end{center}

\begin{abstract} Let $K$ be a unramified $p$-adic field with the absolute Galois group $G_K$ and $T$ a crystalline $\Z_p$-representation of $G_K$. 
We study the graded pieces of integral filtration on $D_\dR(T)$ given by Nygaard filtration of the attached Breuil-Kisin module of $T$. We show that the $i$-graded piece has nontrivial $p$-torsion only if $ i =  r_j +m p$ for a Hodge-Tate weight  $ r_j$ of $T$ and $m$ a positive integer.   
\end{abstract}
\tableofcontents

\section{Introduction}  Let $\kappa$ be a perfect field with $ {\rm char} (\kappa) = p > 0$,  $ \O_K = W(\kappa) $ and $ K= \O_K [\frac 1 p]$. Denote $ G_K : = {\rm Gal} (\overline K/K)$,  $ \fS: = W(\kappa)[\![u]\!]$ and $E(u)= u -p$. Then $ (\fS, (E))$ is a Breuil-Kisin prism if Frobenius $ \varphi$ on $ \fS$ is given by $ \varphi (u) = u^p$.  Let $T$ be a crystalline finite free $\Z_p$-representation of $G_K$ with Hodge-Tate weights $  r_1 , \dots,  r_d$ so that 
$r_1 = 0$, $r_i \leq  r_{i+1} , \forall i = 1, \dots , d-1$ and $h: = r_d$. Let $ \M $ be the Kisin module attached to $T$ (\cite{KisinFcrystal}) and $ \M ^* : = \fS \otimes_{\varphi, \fS }\M$. It is known (see \S \ref{sec:overconvergentsheaf} for more details) that $ \M ^*$ admits Nygaard filtration $ \Fil^i \M^* \subset \M^*$ so that  $ (\M^* / E \M^*) [\frac 1 p]\simeq D_\dR (T) $ and $\ev _p  (\Fil^i \M^*[\frac 1 p]) = \Fil^i D_\dR (T)$ where  $ \ev_p : \M^*[\frac 1 p] \to (\M^* / E \M^*) [\frac 1 p] \simeq D_\dR (T)$ is the projection. Set $ M : = \M^* / E\M^* = \ev_p (\M^*)\subset D_\dR (T)$ and $ \gr^i M := \ev_p (\Fil^i \M^*)/\ev_p (\Fil^{i+1}\M^*). $ Let $ (\gr^iM)_{\tor}$ denote the torsion part of $ \gr^i M$. 
This paper aims to prove the following result. 

\begin{theorem}\label{Thm-1} $\{i | (\gr^i M)_{\tor} \not = 0 \} \subset \{  r_j + m p|  j= 1, \dots , d,  m >0  , r_j + m p \leq h   \}.  $
\end{theorem}
Write $ \Fil^i M : = \ev_p (\Fil^i \M^*)\subset D_\dR(T)$. It is known that $ \Fil^i M [\frac 1 p] = \Fil^i D_\dR(T)$. Thus $ \gr^i M [\frac 1 p]= \gr ^i D_\dR (T)$ which is nonzero if and only if $ i \in \{r_j, j =1 , \dots, d\}$. But nontrivial $ (\gr^iM)_\tor$ could appear beyond $\{r_j\}$. 
So the theorem provides a restriction of $i$ so that $ (\gr^iM)_\tor\not = 0$.   
It turns out that $ (\gr^iM)_\tor$ strongly relates to the shape of $ \Fil^i \M^*$. We call $ \M^*$ has an \emph{adapted basis} if there exists an $ \fS$-basis $ e_1 \dots , e_d$ of $ \M ^*$ so that for each $i$,  $ \Fil^i \M^*$ has an $ \fS$-basis $ \{ E^{a_{ij}} e_j\}$ where $ a_{ij} = \max\{0, i- \max\{\ell| e_j \in \Fil^{\ell}\M^* \setminus \Fil^{\ell +1} \M^*\} \}$.  
 
\begin{proposition}Assume that $(\gr^i M)_\tor= 0 $  for $ i \leq h-1$. Then $\M^*$ has an adapted basis.     
\end{proposition} 
\begin{proof} It is easy to check that  $\gr^i M$ has no nontrivial $p$-torsion for $ i \leq h-1$ if and only if $ \Fil^i M$ is saturated in $ M$, namely $ \Fil^i M = M \cap \Fil ^i M [\frac 1 p]$. Then the proposition follows the same proof of \cite[Prop4.5]{GLS-1}.     
\end{proof}
By Theorem \ref{Thm-1}, if $ h \leq p$ then  $\gr^i$ has no $p$-torsion for $ i \leq h-1$. Now we recover one of the main technical results in \cite[Cor. 4.19]{GLS-1}
 \begin{corollary} If $h \leq p$ then $\M^*$ has an adapted basis.     
 \end{corollary}   
By \cite[Example 6.8]{GLS-1}, there exists a crystalline representation $T$ of $G_{\Q_p}$ with Hodge-Tate weights $ \{0, p+1\}$ so that $ \M^*$ has no adpated basis. This implies that $ (\gr^p M)_\tor\not = 0$.  

 Our result is motivated by the following theorem of Gee and Kisin: Let $ \overline \M: = \M / p\M$. Then for a $\kappa[\![u]\!]$-basis $ \bar e_1 , \dots , \bar e_d$ of $\overline{\M}$, we have $$ \varphi_{\overline \M} (\bar e_1 , \dots , \bar e_d) =  (\bar e_1 , \dots , \bar e_d) X \Lambda Y,  $$
where $X, Y \in \GL_d (\kappa[\![u]\!])$ are invertible matrix and $\Lambda$ is a diagonal matrix with $ u ^{a_j}, j =1 , \dots , d $ on the diagonal. It is known that $ 0\leq a_j \leq h . $
\begin{theorem}[Gee-Kisin]\label{thm-GK}Notations as the above, $ \{a_i \mod p \}= \{r_i \mod p\} $ as multisets. 
\end{theorem}
After the first version of this paper, the author and Hui Gao (\cite{gao2025integralsentheoryintegral}) found the relationships between Theorem \ref{Thm-1} and Theorem \ref{thm-GK}: Both of them are consequences of the \emph{Sen operator} $\Theta$ on $ \M_{\rm HT}: = \M / E \M$. Indeed, Gee and Kisin explained to the author that their method used in Theorem \ref{thm-GK}, which is based on the theory of various stacks in \cite{Bha22} and \cite{BL-b},  can also prove Theorem \ref{Thm-1}. Later Dat Pham \cite{Pham} provided another proof of Theorem \ref{Thm-1} via the Sen operator and the stacky's approach. 
See \cite[\S 1]{gao2025integralsentheoryintegral} for details of the relationships of these papers. 

Instead of using the Sen operator, the strategy of this note is to compute generators of $\Fil^h \M^*$ by the \emph{Griffith transversality} of $ \nabla$ (see Lemma \ref{lem-nabla-integral}).  This allows a much more explicit description of $ \Fil^h \M^*$, which is expected to be crucial, for example, to study the crystalline deformation ring in the future work.



\medskip
\noindent
\textbf{Acknowledgments.}  The author dedicates this paper to the memory of Prof. Xianke Zhang, who was the author's master advisor and introduced the author to the study of algebraic number theory, 

We would like to thank Bargav Bhatt, Dat Pham, Hui Gao and Toby Gee for the discussion and useful comments. 
The author prepared this paper during
his visit in the Institute for Advanced Study and would
like to thank IAS for their hospitality. He is supported by the Shiing-Shen Chern
Membership during his stay at IAS. 

\section{Construction of basis of $\Fil^i \frak{M}^*$}\label{sec:overconvergentsheaf}
\subsection{Preliminary on Nygaard filtration}Recall  that $K = W(\kappa)[\frac 1 p]$ is unramified with $\kappa$ the perfect residue field and ${\rm char}(\kappa) = p>0$. Let $T$ be a crystalline finite free $\Z_p$-representation of $G_K$ with Hodge-Tate weights  $[r_1 , \dots ,  r_d]$, where $ 0 = r_1$, $r_i \leq r_{i +1}$, $ r_d = h >0$ and $d = \dim _{\Q_p} T[\frac 1 p]$.  Let $D = D_\cris (T [\frac 1 p])$ be the filtered $\varphi $-module attached to $T [\frac 1 p]$. Since $K$ is unramifed, we have $ D_{\rm dR} (T [\frac 1 p]) \simeq D$. Write $ \gr ^i D : = \Fil^i D / \Fil^{i +1} D$. Then the set of Hodge-Tate weights ${\rm HT}(T) = \{r_j, 1 \leq j \leq d\} = \{ i \in \Z | \gr ^iD \not = 0 \}$. Let $ (\M, \varphi_\M ) $ be the Breuil-Kisin module attached to $ T$ and write $\M ^* : = \fS \otimes _{\varphi, \fS } \M$. Fix $E = u - p \in \fS$. 
The \emph{Nygaard filtration} of $\M^*$ is defined by 
\begin{equation}\label{Eq-Nygggard}
    \Fil^i \M^* : = \{ x\in \M ^* | (1 \otimes \varphi_{\M}) (x) \in E^i \M \},  
\end{equation}
and set $\gr^i \M^* : = \Fil^i \M ^* / \Fil^{i +1}\M^* $. 

In the following, we review the theory of \emph{Breuil} module which is very useful to understand Nygaard filtration $\Fil^i \M^*$. We refer readers to  \cite[\S 2]{BLL} for more details. 
Let $ S :  = \fS [\![\frac{E^p}{p}]\!]$ and $ \hat S_E $ be the $E$-completion of $W(\kappa)[u][\frac 1 p]$. For any subring $A \subset \hat S _E$, set $ \Fil ^i A : = A \cap E^i \hat S_E $. Clearly, $ \Fil ^i \fS = E^i \fS$, $\Fil^i  S[\frac 1 p] = E^i S [\frac 1 p]$ (but it is not true that $\Fil ^i  S = E^i  S$). Extends Frobenius $\varphi$ on $W(\kappa)$ to $ \fS$ and $ S$ by $ \varphi (u) = u ^p$. 
Set $\nabla :  S \to S$ (resp. $ N : S \to S$) by $ \nabla (f) = \frac{\partial f}{\partial u}$ (resp. $N(f) = \frac{\partial f}{\partial u} u$). The Breuil module over $S[\frac 1 p]$ (associated to $ D= D_\cris(T[\frac 1 p])$) is  $\D : =  S[\frac 1 p]\otimes_K D $, which has following structures: 
\begin{itemize}
    \item Frobenius $\varphi_\D: = \varphi_S \otimes \varphi_D; $
    \item monodromy operator $ \nabla : \D \to \D$ (resp. $N_\D : \D \to \D$) by $ \nabla (s \otimes x) = \nabla (s) \otimes x$ (resp. $ N_\D (x) = N(a) \otimes x + a \otimes N(x)$ \footnote{$N=0$ in the crystalline case considered here.}) where $ s \in S$ and $ x \in D$. Note that $ N_\D= u \nabla $; 
    \item projection map $ \ev_p : \D \to D$ defined by $ \ev_p (s \otimes a) = s(p) a$;  
    \item Filtration $\Fil^i \D \subset \D$ constructed inductively: $ \Fil ^0 \D = \D$ and 
 \begin{equation}
     \Fil^{i + 1} \D : = \{ x \in \Fil^i \D | \nabla (x) \in \Fil^{i}\D, \  \ev_p (x) \in \Fil^{i +1} D\}, 
\end{equation}
\end{itemize}

By the comparison of Kisin module and Breuil module, there exists a canonical isomorphism $\iota:  S[\frac 1 p] \otimes _\fS \M ^* \simeq \D =  S [\frac 1 p ] \otimes_K D $ so that 
\begin{enumerate}
    \item $\iota$ is compatible with $\varphi$-actions on both sides; 
    \item $\iota$ is compatible with filtration on both sides in the  sense that $\iota ( \Fil^i ( S[\frac 1 p] \otimes _\fS \M ^* )) = \Fil ^i \D  $ and filtration on the left sides are defined as follows:  
    \[ \Fil ^i (S [\frac 1 p] \otimes _\fS \M^* ) = \Fil ^i (S [\frac 1 p] \otimes _ {\fS , \varphi}\M): = \{x \in  S [\frac 1 p] \otimes _ {\fS , \varphi}\M| (1\otimes \varphi_\M) (x) \in \Fil^i S [\frac 1 p] \otimes _\fS \M  \}.  \]
\end{enumerate}
So in the following, we identify $  S [\frac 1 p] \otimes_\fS \M ^* $ with $ \D$ via $ \iota$. Consequently, we regard $\M^* $ as $\fS$-submodule of $\D$ and clearly $ \Fil^i \M ^* \subset \Fil^i \D$.  The following lemma collects basic facts on $\Fil^i \M^*$ and $\Fil^i \D$. 
\begin{lemma}\label{lemma-basic} \begin{enumerate}
\item $ \Fil^i \M ^* $ generates $ \Fil^i \D$ as $ S[\frac 1 p]$-module and $ \Fil^i \D \cap \M^* = \Fil^i \M^*$. 
\item For each $i$, $ \Fil^i \M^*$ is a finite free $ \fS$-submodule of $\M^* $ and $ \gr^i \M^*$ is a finite free $ \O_K$-module. 
    \item For each $i$, $E \Fil^{i -1 }\M^* \subset \Fil^i \M^*$ and $E \Fil^{i -1 }\D \subset \Fil^i \D$. 
    \item The map $ \Fil^i \M^* \subset \D \overset{\ev_p}{\to} D $ induces the following short exact sequence 
    \[ 0 \to E \Fil^{i-1} \M^* \to \Fil^i\M^* \overset{\ev_p}{\longrightarrow} \ev_p (\Fil^i \M^*)\to 0. \]
\end{enumerate}
\end{lemma}
\begin{proof} The exactness of the sequence in (4) is equivalent to that $ E \M^* \cap \Fil^i\M^* = E \Fil^{i-1}\M^*$, and this easily follows that construction of $ \Fil ^i \M^* $ in \eqref{Eq-Nygggard}. 
(3) easily follows from the construction of $\Fil^i \M^*$ and $\Fil^i \D $.  \cite[Lem. 4.3]{GLS-1} proves that $ \gr ^i \M^*$ is finite $\O_K$-free and $ \Fil^i \D \cap\M^* = \Fil^i \M^*$. To see that $ \Fil^i \M^*$ is finite free, consider the exact sequence 
\[  0 \to \Fil^ {i +1 } \M/ E \Fil^i \M^* \to \Fil^i\M^* / E \Fil^i \M^* \to \Fil^i \M^* / \Fil^{i +1}\M^* \to 0  . \]
Since $\Fil^ {i +1 } \M/ E \Fil^i \M^* = \ev_p (\Fil^{i+1}\M^*)\subset D$ and $  \Fil^i \M^* / \Fil^{i +1}\M^* = \gr^i \M^*$ are torsion free and hence finite free $\O_K$-modules, we conclude that $ \Fil^i \M^* /E \Fil^i \M^*$ is finite free $\O_K$-module. Since $ E^i \M^* \subset \Fil^i \M^* \subset \M^*$ and $ \M^*$ is finite $\fS$-free, $ \Fil^i \M^*$ is finite $\fS$-free by NAK. 

To prove  (1), first note that $ E^i \D \subset \Fil^i \D$ and $ E^i \M^* \subset \Fil^i \M^*$. Now given a $ x= \sum_j f_j \otimes a_j  \in \Fil^i (S [\frac 1 p] \otimes_{\varphi, \fS}\M)$, with $ f_j \in S [\frac 1 p]$, $ a_j \in \M ^* $, by removing $ E^l $-term with $ l  \geq i$, we may assume that $ f_j \in W(\kappa) [u][\frac 1 p ]$. Therefore $p^n x\in \M^*$ for sufficient large $n$. Note that  $ p ^n x \in \Fil^i (S[\frac1 p]\otimes_{\varphi, \fS}\M^*)$, by construction, we have $ p ^n x \in \Fil^i \M^*$. Therefore, $ x $ is $ S [\frac 1 p]  \Fil^i \M^*\subset \Fil^i \D$ as required.   
\end{proof}
\begin{remark} \label{rmk-Tor}Indeed, $ \Fil^i \D \simeq S [\frac 1 p] \otimes_\fS \Fil^i \M^*$. To show this, it suffices to show that the natural map $ S[\frac 1 p] \otimes_\fS  \Fil^i \M^* \to \D\simeq S[\frac 1 p]\otimes_\fS \M^*$ is injective. By induction on $i$ and using exact sequence $0 \to \Fil^{i +1}\M^* \to \Fil^i \M^* \to \gr^i \M^* \to 0$, this follows that $ {\rm Ext}^1_\fS (S[\frac 1 p], \O_K)= 0$. 
    
\end{remark}

\subsection{The range of $\nabla (\Fil ^i \M^*)$. }



    
Set $ \Fil^i \M^*: = \M^*$  and $\Fil^i \D : = \D $ if $ i < 0$. The following Lemma, which can be regarded as integral Griffith transversality,  is important for later use. 
\begin{lemma}\label{lem-nabla-integral} $ \nabla (\Fil^i \M^*) \in  S_p \otimes _\fS \Fil^{i -1} \M ^*$ where $S_p : = \{f \in S [\frac 1 p]| \ev_p (f) \in W(\kappa)\}$. 
\end{lemma}
\begin{proof} In this Lemma, we allow the base field $K$ to be finite ramified over $ W(\kappa)[\frac 1 p]$. In such generality, $ S_p$ is replaced by 
$ S_\pi : = \{ f\in S[\frac 1 p] | f (\varpi)\in \O_K\}$ for a fixed uniformizer $ \varpi\in \O_K$.  
We use the same idea in \cite[Prop. 4.6]{GLS-1} and \cite[Prop. 2.13]{Liu-phiN-lattice}.  As the above, we fix a uniformizer $ \varpi\in  \O_K $ with Eisenstein polynomial $E \in W(\kappa)[u]$,  a compatible system $\{\varpi_n\}_{n\geq 0}$ of $p^n$-th roots with $\varpi_0=\varpi$ as well as a compatible system $\{\zeta_n\}_{n\geq 0}$ of $p^n$-th roots of $1$. Write $\underline{\varpi}^\flat  \coloneqq (\varpi_n)_{n\geq 0}$ and $\underline\zeta^\flat \coloneqq (\zeta_n)_{n\geq 0}$ as elements in $\O_{\C_p}^\flat$ where $ \C_p$ is the $p$-adic completion of $\overline K$. Embed $\fS \to \Ainf$  and $S\to A_\cris$ given by $u \to [\underline{\varpi}^\flat]$. Let $K _\infty =\bigcup_{n = 1}^\infty K (\varpi _n )$, $L = \bigcup_{n=1}^\infty K_{\infty} (\zeta_{p ^n})$ and $K_{1^\infty}= \bigcup_{n =1}^\infty K (\zeta_{p ^n})$. We may always assume that $ K_\infty \cap K_{1^\infty} =K$ by selecting a suitable uniformizer $\varpi$ (this is only needed for $p=2$ by \cite[Lem.2.1]{WangXY}). 
Let $\tau$ be a topological generator of $\Gal (L/ K_{1 ^\infty}) $. We will also use $\tau$ to denote a lifting of $\tau$ in $G_K$ when it is acting on an element fixed by $\Gal(\overline{K}/L)$. 


Write $T^\vee : = \Hom_{\Z_p} (T, \Z_p)$ the dual of $T$.  Then exists an $\Ainf$-linear \emph{injection} 
\[ \iota_\fS : \Ainf \otimes _\fS \M \longrightarrow T^\vee \otimes_{\Z_p} \Ainf \]
so that $ \iota_\fS$ is compatible with Frobenius on both sides and $g (\Ainf \otimes_\fS \M) \subset \Ainf \otimes_\fS \M , \forall g \in G_K $ by using $G_K$-action from that on $ T^\vee\otimes_{\Z_p} \Ainf $ via $\iota_\fS$. Also $ \iota_\fS$ induces the following commutative diagram  
\[ \xymatrix{ \Ainf \otimes_{\varphi, \fS}\M \ar@{=}[r] & \Ainf \otimes_{\fS} \M^* \ar[d] \ar[r] ^-{\iota^* _\fS} & T^\vee \otimes_{\Z_p} \Ainf\ar[d] \\  
 & \Acris \otimes_S \D \ar[r]^-{\iota_S} & T^\vee \otimes_{\Z_p} \Acris[\frac 1 p ]}
\]
Here $\iota_\fS^* = \Ainf \otimes_{\varphi, \Ainf}\iota_\fS$, $\iota_S = A_\cris[\frac 1 p] \otimes_{ \Ainf} \iota^*_\fS$, and the left vertical arrows are induced by $ \M^* \subset \D$, and all arrows here are injective. In particular, we may regard $ \Ainf \otimes_\fS \M^* $ (resp. $ \Acris \otimes _S \D$) as submodule of $T^\vee \otimes_{\Z_p} \Ainf$ (resp. $ T^\vee \otimes \Acris [\frac 1 p]$) via $ \iota _\fS$ (resp. $\iota_S$). Note that the modules at the right columns have natural $G_K$-actions. As submodules, the left columns are stable under these $G_K$-actions. The $G_K$-action on $ \D$ is explicitly given by the following formula (see formula (6.19) in \cite{LL2021comparison}): 
\begin{equation}\label{Eqn-G-action}
g (x) = \sum_{i = 0} ^\infty \nabla^i (x) \gamma_i ( {g([\underline \varpi^\flat]) - [\underline \varpi^\flat]}), \ \  \forall g \in G_K,\  \forall x \in \D,       
\end{equation}
where $ \gamma_i= \frac{ (\bullet)^i}{i!} $ is $i$-th divided power. In particular, $G_\infty$-acts on $ \D$-trivially. Pick a $ \tau$ so that $\frac{g([\underline \varpi^\flat])}{[\underline \varpi^\flat]}= [\underline \zeta^\flat]$. The above formula is simplified to $ \tau (x) = \sum\limits_{i = 0}^\infty \nabla^i (x) \gamma_i (u ([\underline \zeta^\flat]-1)) $. Write $ w:  = \frac{[\underline \zeta^\flat]-1}{E} \in \Ainf $.  We see that 
 $ \tau (x) \subset \mathcal R [\frac 1 p] \otimes_S \D $ where $ \mathcal R  \subset \Acris$ is the $p$-adic completion of $\fS[w , \{\gamma_i (E)\}_{i \geq 1 }]$.



Write $ \D_{\Acris} :=  \Acris \otimes_S \D$ and $ \M^*_{\Ainf} : = \Ainf \otimes_\fS \M^*$. Set $\Fil^i \D_{A_\cris} : = A_\cris [\frac 1 p] \otimes_ \fS \Fil^i \M^*$ and  $\Fil^i \M^*_{\Ainf} : = \Ainf \otimes_{\fS}\Fil ^i \M^*$. By a similar arguments as in Remark \ref{rmk-Tor}, we see that $ \Fil^i \D_{A_\cris}$, $ \Fil^i \M^*_{\Ainf}$ injects to $ \D_{A_\cris}$, $ \M^*_{\Ainf}$ respectively. We claim that $ g (\Fil^i \M^*_{\Ainf})\subset \Fil ^i \M^*_{\Ainf}, \forall g \in G_K$. To prove the claim, it suffices to show that $$ \Fil^i \M^*_{\Ainf}= F^i : = \{ x \in \M^*_{\Ainf}| (1\otimes \varphi_{\M}) (x) \in E^i (\Ainf \otimes_{\fS}\M)\}. $$ Consider  $  \M^* $ as an $\fS$-submodule  of $  \M$ via $ 1\otimes \varphi_\M$. Then $\Fil^i \M^* = \M^* \cap E^i \M$ and 
$F^i = (\Ainf \otimes_\fS \M^*) \cap E^i (\Ainf \otimes _\fS \M)$. Since $ \Ainf$ is flat over $\fS$, we have $\Fil^i \M^*_{\Ainf} =F^i $ as required.


As in the proof of  \cite[Prop. 4.6]{GLS-1} and \cite[Prop. 2.13]{Liu-phiN-lattice}, for any subring  $ A \subset \Acris[\frac 1 p]$ such that $ \varphi(A) \subset A$,   let $ I^{[i]} A = \{ x \in A| \varphi ^n (x) \in \Fil ^i B_\dR, \forall n \in \mathbb N\},  $ and $ I ^{[n]}: = I ^{[n]}\Ainf= ([\underline \zeta^\flat]-1)^n \Ainf$. 
Now we claim that $ (\tau- 1)^n (\Fil^i \M^* )\subset \Fil^{i-n}\M^* \otimes_\fS I^{[n]}$. To prove this claim, we first prove that  $$ (\tau-1)^n (\Fil^i \D)\subset \Fil^{i-n }\D \otimes_{S[\frac 1 p]}  I^{[n] } \calR [\frac{1}{p}]\subset \Fil^{i-n} \M^* \otimes _\fS I^{[n]} A_\cris[\frac 1 p]. $$
By a similar argument in \cite[(5.1.2)]{LiuT-CofBreuil}, $ (\tau -1) ^n (x) = \sum\limits_{m \geq n} a_{mn} \nabla^m (x) \gamma_m ( u ([\underline \zeta^\flat]-1)) $ with $ a_{mn}\in \Z$. For $ n \leq m \leq i$, 
we have $ \nabla^m (x) \in \Fil^{i -m}\D$ by Griffith transversality $ \nabla  (\Fil^i\D)\subset \Fil^{i-1}\D  $  and $ [\underline \zeta^\flat]-1 = E w $.   Since $ E^{m -n} \Fil^{i -m}\D \subset \Fil^{i -n}\D$,  $ (\tau -1)^n (x) \in  \Fil^{i-n}\D \otimes_{S[\frac 1 p ]} I^{[n]}\calR [\frac 1 p]$. 

Note that $\Fil^i \M^*_{\Ainf} = \Fil^i \M^* \otimes_\fS {\Ainf}$ and $\forall g \in G_K, \  g (\Fil^i \M^*)\subset \Fil^i \M^*_{\Ainf} = \Ainf \otimes_\fS \Fil^i \M^*$. Then using that $\Fil^{i-n} \M^* $ is finite $ \fS$-free, we have 
$$ (\tau-1)^n (\Fil^i \M^*) \subset \Fil^{i -n}\M^* \otimes_\fS ( I^{[n]} A_\cris[\frac 1 p] \cap \Ainf ) = \Fil^{i-n} \M^* \otimes_\fS I ^{[n]}\Ainf .$$ This prove that claim.  

Now for any $x \in \Fil^i\M^*$, we have 
\[\nabla (x) = \sum_{n = 1}^\infty (-1)^{n -1} \frac{(\tau-1)^n}{n   u ([\underline \zeta^\flat]-1)}(x)  \] 
By the claim, we have $$\frac{(\tau-1)^n}{n  u ([\underline \zeta^\flat]-1) }(x)\in \Fil^{i -n}\M^* \otimes_\fS \frac{([\underline \zeta^\flat -1])^{n-1}}{n} \Ainf \subset \Fil^{i -1}\M^* \otimes \frac{w^{n -1}}{n}  \Ainf.$$ Note that $ \ev_\pi : S \to \O_K$ is the same as the canonical projection $ \theta: \Ainf \to \O_{\C_p}$. Now it suffices 
to check that $ \theta ( \frac{w^{n -1}}{n} ) \in \O_{\C_p}$. It turns out that $ v_p (w) = \frac{1}{p-1}$ and this follows that  $ \frac {n-1}{p-1}\geq v_p (n ) $. 
\end{proof}
\subsection{Construction of basis in $\Fil^i \M^*$}
Recall that $M : = \ev_p (\M^*)\subset D$,  $ \Fil ^i M : = \ev_p (\Fil^i \M^*)$ and $ \gr^i M : = \Fil^i M / \Fil^{i +1} M$. Clearly, $ M \subset D$ is a finite free $ \O_K$-lattice in $D$. Note that $ \Fil ^i M \subset \Fil^i D \cap M$ but may not equal. Consequently, $ \gr^i M$ may have $p$-power torsion. 
In the following, by discussion around \cite[Lemma 4.4]{GLS-1}, we can select $ \O_K$-basis $e_1 , \dots , e_d $ of $M$ so that $ \{e _i, \dots , e_d\}  $ forms a $K$-basis of $ \Fil^{r_i}D $. Write $ d_i : = d- \dim_K \Fil^i D $. Then it is easy to check there exists two  $\O_K$-bases $ \{e^{(i )}_j| j = d_i +1, \dots , d\}$ and $ \{f ^{(i)}_j, j = d_i +1, \dots, d\} $ of $ \Fil^i M$ so that $ e^{(i +1)}_j= p ^{n_{ij}} f^{(i)}_j $, for $ j = d_{i +1}+1 , \dots , d$. We can set $ f^{(0)}_j = e^{(0)}_j = e_j, j =1 , \dots , d$.

We regard $\M$ as an $\varphi (\fS)$-submodule of $\M^*$.  There exists a $\varphi (\fS)$-basis $\e_i \in \M$ so that $ \ev _p (\e_j) =e_j$. Set  $$J: =\{r_j + m p  | j =1 , \dots ,  d; m \in \Z_{\geq 0}; 0 \leq r_j + mp\leq h    \} , $$ and 
$$\mathcal J: =\{r_j + m p  | j =1 , \dots ,  d; m >0 ; 0 \leq r_j + mp\leq h    \} . $$
Note that $ \mathcal J \subset J$ and $ J \setminus \mathcal J= \{ r_ j |\  p \nmid r_j- l, \forall l <r_j, l \in J\}$. 
Recall that our main theorem \ref{Thm-1} claims that $ \gr^i M = 0$ if $ i \not \in \mathcal J$. 

\begin{proposition}\label{prop-constuct}
    For each $0 \leq i \leq h$, there exists \[\{\alpha^{(i)}_j, j =1, \dots , d_i, \ \e^{(i)}_j, j= d_i +1 , \dots , d \}\subset \Fil^i \M ^*  \]so that \begin{enumerate}
        \item $\{\alpha^{(i) }_j , \e^{(i)}_j\} $ is an $\fS$-basis of $ \Fil^i \M ^*$; 
        \item $\ev_p (\alpha^{(i)}_j) = 0$, $\ev _p (\e^{(i)}_j) = e_j^{(i)}$; 
        \item For each $ j = 1, \dots , d_i$, we have 
        \begin{equation}\label{Eqn-alpha}
          \alpha^{(i)}_j = \sum_{ (l, k)  \in J_i}  a_{lk}^{(ij)} E^{i - l} \e_k^{(l)}, \ \ \ \text{ where } J_i: = \{(l, k )|\in J, l < i; d_{l}+1\leq k \leq d\}  
        \end{equation} for some  $ a_{lk}^{(ij)}\in \O_K$. 
    \end{enumerate}
\end{proposition}

  We make induction on $ i$ to prove the above proposition. The case of $ i = 0$ is trivial as we can take $ \e_j^{(0)} = \e_j$ (note that $d_0 = 0$). Now assume that the statement holds for $\Fil^{i}\M ^*$, consider the statement for $\Fil^{i+1}\M^*$. The case $ i \in J$ and $ i \not \in J$ makes a big difference. We divide these two situations into two different subsections. 

  \subsubsection{The case $i \not \in J$. }\label{subsub-1} First note in this case, $p \nmid i-l , \forall l \in J, i > l$. Also $ \gr^i D = 0$ and thus $ \Fil^i D = \Fil^{i +1} D$. So $ d_i = d_{i+1}$.  

For $j = d_{i}+1 , \dots , d$,  we construct $\beta_j: =  \e^{(i)}_j + \sum\limits_{(l, k ) \in J_i}  x^{(ij)}_{lk} E^{i-l} \e^{(l)}_{k} $ with $ x_{lk}^{(ij)} \in \O _K$ undetermined so that $ \beta_j  \in \Fil^{i+1} \M^*$. Since $ E^{i-l} \e ^{(l)}_k$ and $ \e^{(i)}_j$ are in $ \Fil^i \M^*$ and $ \ev _p (\beta_j) = \ev_p (\e_j ^{(i)}) = e^{(i)}_j \in \Fil^i D = \Fil^ {i+1}D $,  to construct $ \beta_j \in \Fil^{i +1}\M^*$, it suffices to select $ x_{lk}^{(ij)}$ so that $ \nabla (\beta_j) \in \Fil^i \D$. To simplify the notation, we write $ x_{k l}= x_{kl}^{(ij)}$. 

By Lemma \ref{lem-nabla-integral}, we have $ E \nabla( \e^{(i)}_j) = \sum\limits_{k=1} ^{d_i} c_k \alpha^{(i)}_k  + \sum\limits_{k > d_i} c_k' \e^{(i)} _k$ with $c_k , c'_k \in  S_p $.  Note that $ E \nabla (\e^{(i)}_j )\subset E \D$, then $ \ev_p (E \nabla (\e^{(i)}_j )) = 0$. Since $ \{\ev_p (\e^{(i)}_j)\} $ forms a basis of $ \Fil^i D$, we conclude that $ c'_k\in E S [\frac 1 p]$. Thus, 
 \begin{eqnarray*}
     \nabla( \e^{(i)}_j) &= &\sum\limits_{k=1} ^{d_i} c_k \frac{\alpha^{(i)}_k} { E}    + \sum\limits_{k > d_i} \frac{c_k'}{E}   \e^{(i)}_k  \\ &= & \sum\limits_{k=1} ^{d_i} c_k(p)   \frac{\alpha^{(i)}_k} { E}    + \sum\limits_{k > d_i} (\frac{c_k'}{E})    \e^{(i)}_k   +  \sum\limits_{k=1} ^{d_i} (c_k - c_k(p))    \frac{\alpha^{(i)}_k} { E}.     
 \end{eqnarray*}
Since $c_k - c_k(p) $ is in $E S [\frac 1 p ]$, it suffices to pick $x_{lk}\in \O_K$ so that 
$$\nabla \left (\sum_{(l , k) \in J_{i}} x_{lk} E^{i-l} \e^{(l)}_{k}\right ) + \sum\limits_{k=1} ^{d_i} c_k(p)   \frac{\alpha^{(i)}_k} { E}  \in \Fil^i \D.  $$ By induction on the shape of $\alpha_j^{(i)}$, we can write 
\[ \sum\limits_{k=1} ^{d_i} c_k(p)   \frac{\alpha^{(i)}_k} { E}  = \sum_{(l , k)\in  J_i} b_{lk}E^{i - l-1 } \e_k^{(l)} \]
for some $ b_{lk}\in \O_K$,  which is a  combination of $ a^{(ij)}_{lk}$ and $c_k (p)$ (note that $c_k (p)$ is in $\O_K$ by Lemma \ref{lem-nabla-integral}). Hence it suffices to show that the existence of $ x_{lk}\in \O_K$ so that 
\[\nabla \left (\sum\limits_{(l , k) \in  J_i } x_{lk} E^{i-l} \e^{(l)}_{k}\right ) +  \left (\sum_{(l, k ) \in J_{i}} b_{lk}E^{i - l-1 } \e_k^{(l)} \right ) \in \Fil^i \D.  \]
We prove this by back induction on $ l$.  Let $s = \max \{l \in J|  l \leq i-1\}$. Note that $\nabla ( x_{sk} E^{i-s} \e^{(s)}_{k})=  x_{sk} (i-s) E^{i-s-1 } \e^{(s)}_{k} +  x_{sk} E^{i-s} \nabla (\e^{(s)}_{k}) $. By the calculation in the third paragraph of this subsection \S \ref{subsub-1}, we have 
 $ E^{i -s} \nabla( \e^{(s)}_j) = \sum\limits_{k=1} ^{d_s} c_k E^{i-s-1} \alpha^{(s)}_k  + \sum\limits_{k > d_s} c_k' E^{i-s-1}\e^{(s)} _k$ with $c_k, c'_k \in  S_p $ and $c'_k \in E  S[\frac 1 p]$. Using the induction on $\alpha_k^{(s)}$ and note that $$ \sum\limits_{k=1} ^{d_s} (c_k-c_k (p))  E^{i-s-1} \alpha^{(s)}_k  + \sum\limits_{k > d_s} c_k' E^{i-s-1}\e^{(s)} _k \in \Fil^i\D, $$ we have  $ E^{i -s} \nabla( \e^{(s)}_j) = (\sum \limits _{(l, k ) \in J _{s}} c''_{lk} E^{i -l -1 } \e^{(l)}_k) + z$ with $z \in \Fil^i \D $ and $c''_{lk}\in \O_K $. 
 Therefore, by setting $ x_{sk} = -(i-s)^{-1}b _{sk} $ (note that $p \nmid i-s$ as $ i \not \in J$), we need to further  solve new $x_{lk}\in \O_K$ for 
\[\nabla \left ( \sum \limits_{(l, k ) \in J_{i}, l <s } x_{lk} E^{i-l} \e^{(l)}_{k}\right ) + \left (\sum_{(l, k ) \in J_{i}, l <s} b_{lk}E^{i - l-1 } \e_k^{(l)} - \sum \limits _{l \in J _{s}} (i-s)^{-1}b _{sk} c''_{lk} E^{i -l -1 } \e^{(l)}_k\right ) \in \Fil^i \D.  \]
Continue this step and decrease $l \in J$ until $l =0$. In this situation, $ \nabla (x_{0k} E ^i \e_k^{(0)}) = x_{0k} i E^{i -1} \e^{(0)}_k + x_{0k} E^i \nabla (\e ^{(0)}_k)$. As $  E^i \nabla (\e ^{(0)}_k) \in \Fil^i \D$, $p \nmid i$, $ x_{0k}\in \O_K$ can be always found. This completes the construction of $\beta_j$. 

Now we claim that $\Fil^{i +1}\M^*$ is generated by $ \{E\alpha^{(i)}_j, j =1 , \dots , d_i, \beta_j, j = d_i +1 , \dots , d\}.  $ Write $F^{i+1}\M^*$ be $\fS$-submodule of $ \Fil^i \M^*$ generated by   $ \{E\alpha^{(i)}_j, j =1 , \dots , d_i, \beta_j, j = d_i +1 , \dots , d\}$. By the above construction, $ \beta_j \in \Fil^{i+1 }\M^*$. So $ F ^{i+1}\M ^* \subset \Fil^{i+1}\M^*$. By the construction  $\beta_j =  \e^{(i)}_j +  \sum \limits_{(l, k)  \in J_{i}} x^{(ij)}_{lk} E^{i-l} \e^{(l)}_{k} $, $\{\alpha^{(i)}_j, j =1 , \dots , d_i, \beta_j, j = d_i +1 , \dots , d\}$ is another basis for $ \Fil^i \M^*$. In particular, $ E \Fil ^i \M^* \subset F^{i+1}\M^*$. Since $ \ev_p (\beta_j)= \ev_p (\e_j ^{(i)})= e_j ^{(i)}$, $\ev_p (F^{i+1}\M^*) = \Fil^{i} M$. But  $ F^{i+1}\M^* \subset \Fil^{i+1}\M^*  $ and $ \ev_p (\Fil^{i+1}\M^*) = \Fil^{i+1} M$. This forces that $ \Fil^i M = \ev _p (F^{i+1} \M^*) =\ev _p (\Fil^{i+1}\M^*) = \Fil^{i+1}M. $ Together with that $ \Fil^ {i+1} \M^* / E \Fil^i \M^* = \Fil^{i+1} M$ and that $E \Fil ^i \M^* \subset F^{i+1}\M^*$, we conclude that $F^{i+1}\M^*= \Fil^{i+1}\M^*$ and $ \gr^i M = \Fil^i M /\Fil^{i+1} M = 0$. 
Finally, since  both $ f_j^{(i)} $ and $ e_j^{(i)}$ are bases of $ \Fil^i M$, we may select an invertible matrix $ A \in \GL_{d-d_i} (\O_K)$ so that  $ (\beta'_{d_i +1 }, \dots , \beta'_d)= (\beta_{d_i +1}, \dots , \beta_{d}) A$ satisfies $ \ev_p (\beta'_j) = f^{(i)}_j$. Since $\gr^i M = 0 $, $e^{(i+1)}_j =  f^{(i)}_j$ and we can set $ \e^{(i+1)}_j: = \beta'_j$ and $ \alpha^{(i+1)}_j = E \alpha^{(i)}_j$ as required. 
\subsubsection{The case $i \in J $} First let us consider the following commutative diagram:
\[ \xymatrix{ 0 \ar[r]  & \Fil^i \M^* \ar[d]\ar[r]^- E &  \Fil^{i+1} \M^*\ar[d]\ar[r]^-{\ev_p}  & \Fil^{i+1} M\ar[d] \ar[r]& 0\\
 0 \ar[r]  & \Fil^{i-1} \M^*\ar[d] \ar[r]^-E  &  \Fil^{i} \M^*\ar[r]^-{\ev_p} \ar[d]^q & \Fil^{i} M \ar[r]\ar[d]^{\bar q} & 0\\
  0 \ar[r]  & \gr^{i-1}\M^* \ar[r]^-{\iota_E} &  \gr^{i} \M^*\ar[r] & \gr^{i} M \ar[r] & 0}\]
Note that all rows and columns are short exact. By induction,  $\{\alpha^{(i)}_j, \e^{(i)}_j\}$ is an $ \fS$-basis of $ \Fil^i \M^*$ so that $ \{\ev _p (\e^{(i)}_j) = e^{(i)}_j\}$. Recall that $ \{f^{(i)}_j\}$ is another $\O_K$-basis of $ \Fil^i M$ so that $ e^{(i+1)}_j = p^{n_{ij}} f^{(i)}_j$. 
Similar to the argument at the end of \S \ref{subsub-1}, we may choose $ \beta_j\in \Fil^i\M^*$ so that $ \ev _p (\beta_j) = f^{(i)}_j$ and then $\{\alpha^{(i)}_j, \beta_j\}$ is still an $ \fS$-basis of $ \Fil^i \M^*$. To simplify the notation, write $ \alpha_j = \alpha^{(i)}_j$.

Now let us discuss the structure of $ \gr^i M$. Divide the set $R:= \{j|  j =d_i +1 , \dots , d\}$ into three subsets $R_0$, $R_f$ and $R_{\tor}$ in the following: 
$R_0 : = \{ j \in R | \bar q (f^{(i)}_j)= 0\}$;  $R_f: = \{j \in R| \bar q (f^{(i)}_{j}) \text{ is torsion free} \}$; $R_\tor: = \{j\in R | \bar q (f^{(i)}_j) \text{ is killed by } p ^{n_{ij}}\}.  $ Write $ \bar f_j = \bar q (f^{(i)}_j)$, $ n _{j}= n _{ij}$,   $ W= \O_K = W(\kappa)$ and $ W_n : = W /p ^n W$. 
Now we have 
\[\gr^i M = \bigoplus_{j \in R_f} W \bar f_j \oplus \bigoplus_{j \in R_\tor} W_{n_{j}} \bar f_j. \]
\begin{lemma}\label{lem-choose-basis}
    There exists an $\O_K$-basis $\{\tilde \alpha_j\}$ of $ \gr^{i-1}\M^*$ (as a submodule of $ \gr ^i \M^*$ via $\iota_E$) so that 
    for any $ j \in R_\tor$ there exists a unique $k(j)$ satisfying $\tilde \alpha_{k(j)}= p ^{n_j} q (\beta_j)$. 
\end{lemma}
\begin{proof} Let $ N : =\bigoplus\limits_{j \in R_f} W q(\beta_j) \subset \gr^i \M^* $ and $ N' := \gr^i \M^* / N $. It is easy to check that $ N \cap  \gr^{i-1}\M^* = 0$ and $N'$ is torsion free. Then we have an exact sequence: 
\[ \xymatrix{0 \ar[r]  & \gr^{i-1}\M^* \ar[r]^-E  &  N' \ar[r] & \bigoplus\limits_{j \in R_\tor} W_{n_j} \bar f_j\ar[r] & 0}\]
Consider the natural map $\nu : \bigoplus\limits _{j \in R_\tor} \O_K q(\beta_j) \to \gr^i \M^* \to  N'$ and $ \nu' : N' \to \bigoplus\limits_{j \in R_\tor} W_{n_j} \bar f_j$. For any $ j \in R_{\tor}$, since $ \bar q (\ev_p (\beta_j)) = \bar f_j$, $ \nu' \circ \nu  \mod p$ is an isomorphism,  
this forces $\nu $ is  injective.  Note that $ \bar f_j$ is killed by $p ^{n_j}$, we have $ p ^{n_j} q (\beta_j) \in \gr^{i-1}\M^*$. Write $ \alpha'_j:  = p ^{n_j} q (\beta_j)$. 
Then we have the following commutative diagram
\[ \xymatrix{ 0 \ar[r] & \bigoplus\limits_{j \in R_\tor} \O_K \alpha'_j \ar[r]\ar@{_{(}->}[d] & \bigoplus\limits _{j \in R_\tor} \O_K q(\beta_j)  \ar[r]\ar@{_{(}->}[d]^\nu & \bigoplus\limits_{j \in R_\tor} W_{n_j} \bar f_j \ar[r]\ar[d]^\wr  & 0 \\  0 \ar[r]  & \gr^{i-1}\M^* \ar[r]^-E  & N' \ar[r]^-{\nu'} &  \bigoplus\limits_{j \in R_\tor} W_{n_j} \bar f_j \ar[r] & 0 .  }  \]    
Consider modulo $p$ for the second row, we have $N' / pN ' \simeq N'' \oplus \bigoplus\limits _{j \in R_\tor} W_1 \tilde f_j  $ where $ \tilde f_j = f_j \mod p$ and $N''$ is the image of $ \gr ^{i -1}\M^* / p$. Pick a lift $\alpha''_l  \in \gr^{i-1}\M^*$ so that the image of $ \alpha''_l  $ forms a $\kappa$-basis of $ N''$. By Nakayama's lemma, we easily check that $ \{\alpha''_l, \alpha'_j \}$ forms an $\O_K$-basis of $ \gr^{i-1}\M^*$. Now just reindex $\{ \alpha''_l, \alpha'_j\}$ to $ \tilde \alpha_j, j =1 , \dots , d_i$ and the $\{\tilde \alpha_j\}$ is required basis of $\gr^{i-1}\M^*$. 
\end{proof}
Let $A \in \GL_{d_i} (\O_K)$ so that $ (\tilde \alpha_1, \dots , \tilde \alpha_{d_i}) = q(\alpha^{(i)}_1 , \dots , \alpha_{d_i}^{(i)}) A$ and set $ (\hat \alpha_1 , \dots , \hat \alpha_{d_i }) = (\alpha^{(i)}_1 , \dots , \alpha_{d_i}^{(i)})A$. 
 For each $j\in R_\tor$, set $\hat \beta_j =   p ^{n_j}\beta_j - \hat \alpha_{k(j)}$. Since $ q (\hat \beta_j) = 0 $ by the above lemma,  
$ \hat \beta _j \in \Fil^{i+1} \M ^*$. It is clear that  $ \ev_p (\hat \beta_j) = \ev_p (p ^{n_j} \beta_j ) = p ^{n _{ij}} f^{(i)}_j = e^{(i+1)}_j$. 

Now we define $ \alpha^{(i+1)}_j: = E \hat \alpha^{(i)}_j$ if $ j\in \{1, \dots , d_i\}\setminus \{k (j) | j \in R_\tor\} $; $ \alpha^{(i+1)}_{k(j)} :  = E \beta _j $ if $ j \in R _\tor$ and $ \alpha_j^{(i+1)}= E \beta_j$ if $ j \in R_f$; Define $ \e^{(i+1)}_j:  = \hat \beta_j$ for $j \in R_\tor$. For $ j \in R_0$, choose a lift $\e^{(i+1)}_j\in \Fil^{i+1}\M^*$ so that $ \ev_p (\e^{(i+1)}_j)= e^{(i+1)}_j$. 

Now we need to check that $\{\alpha_j^{(i+1)}, \e_j^{(i+1)}\}$ satisfies the requirement of Proposition \ref{prop-constuct} to complete the induction. First note that $ R_\tor\cup R_0$ is the set of indices $j$ for $ e^{(i+1)}_j$. So we have $ \ev _p (\e ^{(i+1)}_j)= e^{(i+1)}_j$, which forms a basis of $\Fil^{i+1}M . $

Next we check that $ \alpha_j^{(i+1)}$ satisfies the requirement in Equation \eqref{Eqn-alpha}.  If $ j\in \{1, \dots , d_i\}\setminus \{k (j) | j \in R_\tor\} $ then  $ \alpha^{(i+1)}_j: = E \hat \alpha^{(i)}_j$ and this follows the induction on $i$ and that $ \hat \alpha^{(i)}_j$ is $\O_K$-linear combination of $ \alpha^{(i)}_j$. Note that $\beta_j$ is also $\O_K$-linear combination of $ \e^{(i)}_j$ and $i\in J$. So  $ \alpha^{(i+1)}_{k(j)}   = E \beta _j $ for $ j \in R _\tor$ and $ \alpha_j^{(i+1)}= E \beta_j$ for $ j \in R_f$ still satisfy the requirement in Equation \eqref{Eqn-alpha}.

Finally, we check that  $\{\alpha_j^{(i+1)}, \e_j^{(i+1)}\}$ is an $\fS$-basis of $ \Fil^{i+1}\M^*$. First,  it is clear that $ \{\alpha_j^{(i+1)}, \e_j^{(i+1)}\} \subset \Fil^{i+1}\M^*$ by construction. 
Set $F^{i+1}\M^*$ be the $\fS$-submodule of $\Fil^{i+1}\M^*$ generated by $\{\alpha_j^{(i+1)}, \e_j^{(i+1)}\}$. Since $ \ev_p (\e_j^{(i+1)})= e_j ^{(i+1)}$ forms a basis of $ \Fil^{i+1} M$, by using the same argument in the end of \S \ref{subsub-1}, it suffices to show that $ E \Fil^i \M ^* \subset F^{i+1} \M^*$. Equivalently, we have to show that $ E \hat \alpha_j $  and $ E \beta_j $ are in $F^{i+1}\M^*$. From our construction, this is clear for $E \hat \alpha^{(i)}_j$ if $ j\in \{1, \dots , d_i\}\setminus \{k (j) | j \in R_\tor\} $,  $E \beta _j $ if $ j \in R _\tor$ and $  E \beta_j$ if $ j \in R_f$. For $j \in R_\tor$, since $ \e_j^{(i+1)} = \hat \beta_j = p ^{n_j} \beta_j - \hat \alpha_{k (j)}$. Then $ E \hat \alpha_{k(j)} = p ^{n_j} E\beta_j- E \e^{(i+1)}_j= p ^{n_j}\alpha^{(i+1)}_{k(j)} - E \e^{(i+1)}_j \in F^{i+1}\M^*. $ For $ j \in R_0$, note that $ \ev_p (\e^{(i+1)}_j ) = e^{(i+1)}_j =  \ev_p (\beta_j )$, we have $ \e ^{(i+1)}_j - \beta_j \in E \Fil^i \M^*$. Thus it is easy to check that $ \Fil^{i}\M^*$ has an $\fS$-basis $ \{\hat \alpha_j, j=1 , \dots d_i;  \beta_j, j \not \in R_0;  \e^{(i+1)}_j, j \in R_0 \}$. 
Hence $ E\beta_j$ is a linear combination of $ \{E\hat \alpha_j, j=1 , \dots d_i;  E\beta_j, j \not \in R_0;  E\e^{(i+1)}_j, j \in R_0 \}$ and hence in $ F^{i+1}\M^*$,  as required. 
\subsection{The proof of Theorem \ref{Thm-1}} 
\subsubsection{The case $i \not\in J$.} In the last paragraph of \S \ref{subsub-1}, we see that if $ i \not \in J$ then $ \Fil^{i +1}\M ^*$ is generated by 
$\{ E\alpha_i, \beta_j \}$ and  $ \ev_p (\beta_j) = e_j^{(i)} $ for $ j = d_i+1 , \dots , d$. Thus $ \ev_p (\Fil^{i+1}\M^*) = \Fil^i M$ and $ \gr^i M = 0$. 
\subsubsection{The case $i \in J \setminus \mathcal J$.} Note that $ i = r_j$ but $ p \nmid  r_j - l, \forall l < r_j, l \in J$. We can select an invertible matrix $A \in \GL_{d-d_i} (\O_K)$ so that if we replace $ \e^{(i)}_{d_i +1} , \dots , \e^{(i)}_d$  by $  (\e^{(i)}_{d_i +1} , \dots , \e^{(i)}_d) A$ then $ \ev_{p} (\e^{(i)}_{d_{i+1} +1}), \dots ,  \ev_p (\e^{(i)}_{d})$ is a basis of $ \Fil^{i+1} D$. 
As in \S \ref{subsub-1}, for $j = d_{i+1}+1 , \dots , d$,  we construct $\beta_j: =  \e^{(i)}_j +  \sum \limits_{ (l, k)  \in J_{i}} x^{(ij)}_{lk} E^{i-l} \e^{(l)}_{k} $ with $ x_{lk}^{(ij)} \in \O _K$ undetermined so that $ \beta_j  \in \Fil^{i+1} \M^*$. Since $ E^{i-l} \e ^{(l)}_k$ and $ \e^{(i)}_j$ are in $ \Fil^i \M^*$ and $ \ev _p (\beta_j) = \ev_p (\e_j ^{(i)}) $ is a basis for $\Fil^ {i+1}D $,  to construct $ \beta_j \in \Fil^{i +1}\M^*$, it suffices to select $ x_{lk}^{(ij)}$ so that $ \nabla (\beta_j) \in \Fil^i \D$. Note that all arguments in \S \ref{subsub-1} go through because the key assumption $ p \nmid i -s$ still holds here as $ s\in J, s < i$. Therefore, we construct $ \beta_j \in \Fil^{i+1}\M^* $ for $ j = d_{i+1}+1 , \dots , d$. Now $ \Fil^{i+1} M = \ev_p (\Fil^{i+1}\M^*)$ contains $ \ev_p (\beta_j) = e^{(i)}_j$  for $ j = d_{i+1}+1 , \dots , d$. Let 
$F^{i+1} M\subset \Fil^{i+1} M$ be the finite free $\O_K$-submodule generated by  $e^{(i)}_j$  for $ j = d_{i+1}+1 , \dots , d$. Then $\Fil^i M /F^{i+1} M$ is finite free $\O_K$-module with basis $ e_{d_i+1}, \dots,e _{d_{i+1}}$. Together with the fact that $ F^{i+1}M [\frac 1 p] = \Fil^{i+1} D$, we have 
$ F^{i+1} M = \Fil^i M \cap\Fil^{i+1} D$. This forces that $ F^{i+1}M = \Fil^{i+1} M = \Fil^i M \cap \Fil^{i+1} D$. Thus $ \gr^i M$ has no $p$-torsion as required. 

 
\bibliographystyle{amsalpha}
\bibliography{mybib}

\end{document}